\documentclass{amsart}

\usepackage{geometry}
\usepackage{amsmath}
\usepackage{amssymb}
\usepackage{amsthm}
\usepackage{cleveref}
\crefname{subsection}{subsection}{subsections}
\crefname{lem}{lemma}{lemmas}
\usepackage{mathtools}
\usepackage{tikz}
\usepackage{comment}
\usetikzlibrary{patterns}
\usetikzlibrary{decorations.pathreplacing}

\usepackage{float}

\theoremstyle{definition}
\newtheorem{thm}{Theorem}[section]
\newtheorem{cor}[thm]{Corollary}

\newtheorem{lem}[thm]{Lemma}

\newtheorem{quest}[thm]{Question}

\newtheorem{exmp}[thm]{Example}

\newtheorem{obs}[thm]{Observation}

\newtheorem{rmk}[thm]{Remark}

\newtheorem{claim}[thm]{Claim}
\newtheorem*{ack*}{Acknowledgements}

\newcommand{\co}{\operatorname{co}}

\allowdisplaybreaks

\Crefname{subsection}{Section}{Sections}

\begin{document}

\title{Sharp stability of Brunn-Minkowski for homothetic regions}
\author{Peter van Hintum \and Hunter Spink \and Marius Tiba}
\email{pllv2@cam.ac.uk, hspink@math.harvard.edu, mt576@dpmms.cam.ac.uk}
\thanks{The first and second author would like to thank their respective institutions Clare College, University of Cambridge and Harvard University. The third author is grateful for the Trinity Hall research studentship for providing research funding.}

\begin{abstract}
We prove a sharp stability result concerning how close homothetic sets attaining near-equality in the Brunn-Minkowski inequality are to being convex. In particular, resolving a conjecture of Figalli and Jerison, we show there are universal constants $C_n,d_n>0$ such that for $A \subset \mathbb{R}^n$ of positive measure, if $|\frac{A+A}{2}\setminus A| \le d_n |A|$, then $|\co(A)\setminus A| \le C_n |\frac{A+A}{2}\setminus A|$ for $\co(A)$ the convex hull of $A$.
\end{abstract}

\maketitle

\section{Introduction}

Let $A,B\subset \mathbb{R}^n$ be measurable sets, which we will always assume throughout to have positive measure. The Brunn-Minkowski inequality states that $$|A+B|^{\frac{1}{n}} \ge |A|^{\frac{1}{n}}+|B|^{\frac{1}{n}}$$ for $|\cdot |$ the outer Lebesgue measure. Equality is known to hold if and only if $A$ and $B$ are homothetic copies of the same convex body (less a measure $0$ set). A natural question is whether this inequality is stable: if we are close to equality in the Brunn-Minkowski inequality, are $A$ and $B$ close to homothetic copies of the same convex body? More precisely, we want to know if $$\omega=\min_{\substack{K_A \supset A, K_B\supset B\\K_A,K_B\text{ homothetic convex sets}}}\max\left\{\frac{|K_A\setminus  A|}{|A|},\frac{|K_B \setminus B|}{|B|}\right\}$$ is bounded above in terms of the quantities $$\delta'=\frac{|A+B|^{\frac{1}{n}}}{|A|^{\frac{1}{n}}+|B|^{\frac{1}{n}}}-1,\text{ and }t=\frac{|A|^{\frac{1}{n}}}{|A|^{\frac{1}{n}}+|B|^{\frac{1}{n}}}.$$ The bound should be positively correlated with $\delta'$, and negatively correlated with $\min(t,1-t)$ (as when $\min(t,1-t)$ is smaller the volumes of $A,B$ are more disproportionate).

In this paper we prove the following sharp stability result for the Brunn-Minkowski inequality in the particular case that $A,B$ are homothetic sets. Taking $A=B$ resolves a conjecture of Figalli and Jerison \cite{Semisum}.

\begin{thm}
\label{introthm}
For all $n \ge 2$, there is a (computable) constant $C_n'>0$ and (computable) constants $d_n(\tau)>0$ for each $\tau \in (0,\frac{1}{2}]$ such that the following is true. With the notation above, if $\tau \in (0,\frac{1}{2}]$ and $A,B\subset \mathbb{R}^n$ are measurable homothetic sets such that $t \in [\tau,1-\tau]$ and $\delta' \le d_n(\tau)$, then 
$$\omega\le C_n'\tau^{-1}\delta'.$$
\end{thm}
For these optimal exponents, we also show that  $e^{\Omega(n)}\le C_n' \le e^{O(n\log n)}$ with explicit constants. We discuss this further in \Cref{CnSharp}.


 The most general stability result for the Brunn-Minkowski inequality was proved in a landmark paper by Figalli and Jerison \cite[Theorem 1.3]{BM}.
There they showed that for arbitrary measurable sets $A,B\subset \mathbb{R}^n$, there exists (computable) constants $\beta_n,C_n>0$ and $\alpha_n(\tau),d_n(\tau)>0$ for each $\tau \in (0,\frac{1}{2}]$ such that if $t \in [\tau,1-\tau]$ and $\delta' \le d_n(\tau)$ then
$$\omega \le C_n\tau^{-\beta_n}\delta'^{\alpha_n(\tau)}$$ (prior to this result, Christ \cite{Christ} had proved a non-computable non-polynomial bound involving $\delta'$ and $\tau$ via a compactness argument). A natural question is therefore to find the optimal exponents of $\delta'$ and $\tau$, prioritized in this order. This question, with $A,B$ restricted to various sub-classes of geometric objects, is the subject of a large body of literature. These optimal exponents potentially depend on which class of objects is being considered. For arbitrary measurable $A,B$ the question is still wide open. Our result is the first sharp stability result of its kind which does not require one of $A,B$ to be convex. 

Most of the literature focuses on upper bounding a measure closely related to $\omega$ for how close $A,B$ are to the same convex set, namely the \emph{asymmetry index} \cite{Figalli10amass}
$$\alpha(A,B)=\inf_{x \in \mathbb{R}^n}\frac{|A\mathrel{\Delta} (s\cdot \co(B)+x)|}{|A|}$$
where $\co(B)$ is convex hull of B, and $s$ satisfies $|A|=|s\cdot \co(B)|$. We always have $\alpha(A,B) \le 2\omega$, so bounding the asymmetry index is weaker than bounding $\omega$.
When $A$ and $B$ are convex, the optimal inequality $\alpha(A,B)\le C_n\tau^{-\frac{1}{2}}\delta'^{\frac{1}{2}}$ was obtained by Figalli, Maggi, and Pratelli in \cite{Figalli09,Figalli10amass}. When $B$ is a ball and $A$ is arbitrary, the optimal inequality $\alpha(A,B)\le C_n\tau^{-\frac{1}{2}}\delta'^{\frac{1}{2}}$ was obtained by Figalli, Maggi, and Mooney in \cite{Euclidean}. We note that this particular case is intimately connected with stability for the isoperimetric inequality. When just $B$ is convex the (non-optimal) inequality 
$\alpha(A,B)\le C_n\tau^{-(n+\frac{3}{4})}\delta'^{\frac{1}{4}}$ was obtained by Carlen and Maggi in \cite{Oneconvex}. Finally, Barchiesi and Julin \cite{Barchiesi} showed that when just $B$ is convex, we have the optimal inequality $\alpha(A,B)\le C_n\tau^{-\frac{1}{2}}\delta'^{\frac{1}{2}}$, subsuming these previous results.

In \cite{Fig} Figalli and Jerison gave an upper bound for $\omega$ when $A=B$, and later in \cite{Semisum} they conjectured the sharp bound $\omega \le C_n\delta'$ when $A=B$. This conjecture was proved in \cite{Semisum} for $n \le 3$ using an intricate analysis which unfortunately does not extend beyond this case. Later on, Figalli and Jerison suggested a stronger conjecture that $\omega \le C_n\tau^{-1}\delta'$ for $A,B$ homothetic regions, which we will prove in this paper.

Because previous sharp exponent results have taken at least one of $A,B$ to be convex, allowing for the use of robust techniques from convex geometry, the implicit hope was that solving this special case would shine a light on the general case where previous methods are not as applicable. The methods we use are indeed very different from the ones from convex geometry and, after an initial reduction, from \cite{Semisum}. We hope that these new techniques, particularly the fractal and the boundary covering detailed in \Cref{Outline}, can provide new insight into finding optimal exponents for general $A,B$.

\subsection{Main theorem}
As we are considering homothetic regions $A,B$, we can replace $A$ with $tA$ and $B$ with $(1-t)A$. Note that $t$ retains its earlier meaning as $t=\frac{|tA|^{\frac{1}{n}}}{|tA|^{\frac{1}{n}}+|(1-t)A|^{\frac{1}{n}}}.$
 Define the \emph{interpolated sumset} of $A$ as $$D(A;t):=tA+(1-t)A=\{ta_1+(1-t)a_2\mid a_1,a_2 \in A\}.$$
Note that we always have $A\subset D(A;t)$. To quantify how small $D(A;t)$ is, we introduce the expression $$\delta(A;t):=|D(A;t)\setminus A|.$$ 
As a further simplification, we note that
$$\delta'=\frac{|D(A;t)|^{\frac{1}{n}}}{|A|^{\frac{1}{n}}}-1=\left(\frac{1}{n}+o(1)\right)\left(\frac{|D(A;t)|}{|A|}-1\right)=\left(\frac{1}{n}+o(1)\right)\frac{\delta(A;t)}{|A|}$$ where $o(1)$ depends on the upper bound on $\delta'$. Since the exponent of $\delta'$ is always at most $1$ (as shown by \Cref{sharpexmp}), we may work with $\frac{\delta(A;t)}{|A|}$ in place of $\delta'$ by absorbing the $\frac{1}{n}+o(1)$ term into $C_n'$ to make a new constant $C_n$.

The following is the specialization of \cite[Theorem 1.3]{BM} to homothetic $A,B$, which we restate for the reader's convenience. 
\begin{thm}\label{BrunnMinkCor}(Figalli and Jerison \cite{BM}) For $n \ge 2$ there are (computable) constants $\beta_n,C_n>0$, and (computable) constants $\alpha_n(\tau),d_n(\tau)>0$ for each $\tau \in (0,\frac{1}{2}]$, such that the following is true. If $A\subset \mathbb{R}^n$ is a measurable set, $\tau \in (0,\frac{1}{2}]$ and $t \in [\tau,1-\tau]$, then
$$|\co(A)\setminus A| \le C_n |A|\tau^{-\beta_n}\left(\frac{\delta(A;t)}{|A|}\right)^{\alpha_n(\tau)}$$ whenever $\delta(A;t) \le d_n(\tau)|A|$.
\end{thm}

Our main result optimizes the exponents to be $\alpha_n=\beta_n=1$ in \Cref{BrunnMinkCor}, verifying the conjecture of \cite{Semisum} and the further generalization to homothetic sets suggested by Figalli and Jerison.

\begin{thm}(\Cref{introthm} reformulated)
\label{mainthm}
For all $n\ge 2$, there is a (computable) constant $C_n>0$ (we can take $C_n=(4n)^{5n}$), and (computable) constants $\Delta_n(\tau)>0$ for each $\tau \in (0,\frac{1}{2}]$ such that the following is true. If $A\subset \mathbb{R}^n$ is a measurable set, $\tau\in (0,\frac{1}{2}]$
and $t \in [\tau,1-\tau]$, then
$$|\co(A)\setminus A| \le C_n\tau^{-1}\delta(A;t)$$
whenever $\delta(A;t)\le \Delta_n(\tau)|A|$.
\end{thm}
\begin{exmp}\label{sharpexmp}
To see that the exponents on $\delta$ and $\tau$ are sharp, suppose we have some inequality of the form $$|\co(A)\setminus A| \le C_n|A| \tau^{-\rho_1}(\delta(A;t)|A|^{-1})^{\rho_2}$$ whenever $\delta(A;t) \le \Delta_n(\tau)|A|$. Take $A=\{(0,0)\cup [\lambda,1+\lambda]\times [0,1]\}\times [0,1]^{n-2},$ with $\lambda \ll  \frac{\Delta_n(\tau)}{2\tau}$, and $t=\tau$. The inequality then becomes
$\frac{\lambda}{2}\le C_n \tau^{-\rho_1}(\tau\lambda(2-3\tau))^{\rho_2}.$
Because we can take $\lambda$ arbitrarily small, it follows that $\rho_2 \le 1$, so $\rho_2=1$ would be the optimal exponent. Given $\rho_2=1$, we then have $\rho_1 \ge 1$, so $\rho_1=1$ would be the optimal exponent.
\end{exmp}

\begin{rmk}
When $n=1$, Theorem 1.1 from \cite{Fig} (a corollary of Freiman's $3k-4$ theorem \cite{Freiman}) with $A$ replaced with $tA$ and $B$ replaced with $(1-t)A$ shows that the optimal exponents are actually $\tau^0\delta(A;t)^1$ in contrast to the case $n \ge 2$.
\end{rmk}

\begin{exmp}
Given exponents $\rho_1=\rho_2=1$, the constant $C_n$ grows at least exponentially as shown by the following example.
Let $R\geq 2$. Consider the set $A\subset \mathbb{R}^n$,
$A=[0,2]^{n-1}\times[-R,0]\cup\{(0,\dots,0,2)\}$.
Then $\co(A)=A\cup\bigcup_{x\in[0,2]} [0,2-x]^{n-1}\times\{x\}$ and $\frac{A+A}{2}=A\cup [0,1]^n$. Hence, $\delta(A,\frac12)=1$ and $|\co(A)\setminus A|=\int_{x\in[0,2]} (2-x)^{n-1}dx=\frac{2^n}{n}$. This example shows that $C_n\geq \frac{2^{n-1}}{n}$.
\end{exmp}

\subsection{Outline}
\label{Outline}By replacing $t$ with $1-t$ we may assume that $t \le \frac{1}{2}$.

\subsubsection{Initial Reduction} We first carry out a straightforward reduction along the lines of the reduction in \cite{Semisum} to \cite[Lemma 2.2]{Semisum}, reducing to the case that $\co(A)$ is a simplex $T$, so $A$ contains all of the vertices of $T$. In this reduction we use \Cref{BrunnMinkCor}, though we need only the following much weaker statement due to Christ \cite{Christ}: $|\co(A)\setminus A||A|^{-1}$ is bounded above by a (computable) function of the parameters $\delta(A;t)|A|^{-1}$ and $\tau$ which, for fixed $\tau$, tends to $0$ as $\delta(A;t)|A|^{-1}$ tends to $0$. 

\subsubsection{Fractal Structure} Next we show that if $\delta(A;t)|A|^{-1}$ is small, then  $A$ contains an approximate fractal structure. For each $i$ we recursively construct a nested sequence of families of simplices $\mathcal{T}_{i,0}\subset \mathcal{T}_{i,1}\subset \ldots$; each family $\mathcal{T}_{i,k}$ consists of translates of $(1-t)^iT$ contained inside $T$, and in the limit $\cup_k \mathcal{T}_{i,k}$ is dense among the translates of $(1-t)^iT$ contained inside $T$. 
We show that there exist universal constants $c_{i,k,n}=i+2k$ such that for translates $T'\in\mathcal{T}_{i,k}$,
$$|((1-t)^iA)_{T'}\cap A| \ge |((1-t)^iA)_{T'}|-c_{i,k,n}\delta(A;t),$$
where $((1-t)^iA)_{T'}$ is the translate of $(1-t)^iA$ induced by the translation that identifies $(1-t)^i T$ with $T'$. Though we need this fractal structure in order to prove this inequality recursively, we only use the corollary that $|T'\cap  A|\ge \frac{|T'|}{|T|}|A|-c_{i,k,n}\delta(A;t)$. This corollary quantitatively establishes that $A$ becomes more homogeneous in $T$ as $\delta(A;t)|A|^{-1} \to 0$.

\subsubsection{Covering a thickened $\partial T$ with small total volume}
Next, we consider a large homothetic scaled copy $R:=(1-\zeta)T$ inside $T$ for $\zeta \approx \frac{1}{n^4}$ and we produce a cover $\mathcal{A} \subset \mathcal{T}_{i,k}$ of $T \setminus R$ for $i\approx 5\log(n)/t$ and $k\approx n \log(n)/t$. The cover $\mathcal{A}$ consists of translates of $(1-t)^i T \approx \frac{1}{n^5}T$ and has the property that the size of $\mathcal{A}$ is at most $(2n)^{5n}$ and the total volume of the simplicies in $\mathcal{A}$ is less than $\frac{1}{2}|T|$. We note that $|\mathcal{A}|,i,k$ affect the complexity of $C_n$, whereas $\zeta$ affects only the complexity of $\Delta_n(\tau)$ and not $C_n$.

In order to produce the covering $\mathcal{A}$ above we proceed in two steps. First, we use a covering result of Rogers \cite{Rogers} to produce an efficient covering $\mathcal{B}$ of $T\setminus R$ with translates of $n^{-\frac{1}{n}}(1-t)^iT$ contained inside $T$. The covering $\mathcal{B}$ has the property that the size of $\mathcal{B}$ is at most $(2n)^{5n}$ and the total volume of the simplicies in $\mathcal{B}$ is less than $\frac{1}{2n}|T|$. Second, we show that for each translate $T'$ of $n^{-\frac{1}{n}}(1-t)^iT$ contained inside $T$, there exists a simplex $T'' \in \mathcal{T}_{i,k}$ such that $T' \subset T''$. This naturally gives the desired cover $\mathcal{A}$.




\subsubsection{Putting it all together}

We may assume that $R \subset D(A;t)$ since a straightforward argument shows this holds whenever $|T\setminus A||A|^{-1}$ is sufficiently small, and $|T\setminus A||A|^{-1}\to 0$ as $\delta(A;t)|A|^{-1} \to 0$ by \Cref{BrunnMinkCor}. Rephrasing the homogeneity statement for $A$, for each $T' \in \mathcal{T}_{i,k}$ we have $$|T'\setminus A| \le \frac{|T\setminus A|}{|T|}|T'|+c_{i,k,n}\delta(A;t).$$
Because $\mathcal{A}$ covers $T\setminus R$ and $A \subset D(A;t)$, we have $|T\setminus D(A;t)|\le \sum_{T' \in \mathcal{A}}|T'\setminus A|$, and by construction $\sum_{T' \in \mathcal{A}}|T'|\le \frac{1}{2}|T|$. Combining these facts, we immediately deduce
$$|T\setminus D(A;t)| \le \frac{1}{2}|T\setminus A|+|\mathcal{A}|c_{i,k,n}\delta(A;t),$$
i.e.
$$|T\setminus A| \le 2(1+|\mathcal{A}|c_{i,k,n})\delta(A;t).$$
Because $|\mathcal{A}| \le (2n)^{5n}$ and $c_{i,k,n} \approx \frac{n\log(n)}{t}$, we see that with $C_n= (4n)^{5n}$ we have
$$|T\setminus A| \le C_n \tau^{-1}\delta(A;t).$$

\section{Initial Reduction}

In this section, we will reduce \Cref{mainthm} to \Cref{MAIN}, similar to the initial reduction in \cite{Semisum} to \cite[Lemma 2.2]{Semisum}.

\begin{thm}
\label{MAIN}
For all $n \ge 2$ there are (computable) constants $C_n>0$ (we can take $C_n=(4n)^{5n}$) and constants $0<\delta_n(\tau)<1$ for each $\tau \in (0,\frac{1}{2}]$ such that the following is true. Let $\tau \in (0,\frac{1}{2}]$, $t \in [\tau,1-\tau]$, and suppose $T\subset \mathbb{R}^n$ is a simplex with $|T|=1$, $A\subset T$ a measurable subset containing all vertices of $T$, and $|A|=1-\delta$ with $0<\delta\le \delta_n(\tau)$. Then
$$|T\setminus A| \le C_n\tau^{-1}\delta(A;t).$$
\end{thm}

 We first need the following geometric lemma.

\begin{lem}
\label{JohnClaim}
For every convex polytope $P$, there exists a point $o \in P$ (which we set to be the origin) such that the following is true. For any constant $b_n(\tau)\in (0,1)$, there exists a  constant $\epsilon_n(\tau)$ such that for any $A \subset P$,  if $t \in [\tau,1-\tau]$ and $|P\setminus A| \le \epsilon_n(\tau)|P|$, then $(1-b_n(\tau))P\subset D(A;t)$.
\end{lem}

\begin{proof}
We may assume that $t \le \frac{1}{2}$ as the statement is invariant under replacing $t$ with $1-t$. Without loss of generality we may assume that $|P|=1$. By a lemma of John \cite{John}, after a volume-preserving affine transformation, there exists a ball $B\subset P$ of radius $n^{-1}$. Denote $o$ for the center of $B$, and set $o$ to be the origin.

We will show that $(1-b_n(\tau))P \subset D(A;t)$. Take $x \in (1-b_n(\tau))P$, and let $y$ be the intersection of the ray $ox$ with $\partial P$. Note that the ratio $r=|xy|/|oy| \ge b_n(\tau)$.



Let $H$ be the homothety with center $y$ and ratio $r$. This homothety sends $o$ to $x$ and $P$ to $H(P)$.
Note that $H(P) \subset P$ because $P$ is convex. Denoting
$$A'=A\cap H(P),$$ we have
$$|A'|\ge r^n-\epsilon_n(\tau).$$
The statement
$x \in D(A';t)$ is implied by the statement that
$o \in D(C;t)$ for $C=H^{-1}(A')\subset P$, which we will now show (in fact we will show $o\in D(C\cap B;t)$). 

Note that  $|C|\ge 1-r^{-n}\epsilon_n(\tau)$, so $|B\setminus C| \le r^{-n}\epsilon_n(\tau)$. Consider the negative homothety $H'$ scaling by a factor of $-\frac{t}{1-t}\in [-1,0)$ about $o$. If $o \not \in D(C \cap B;t)$, then at least one of $y$ and $H'(y)$ is not in $C \cap B$ for every $y \in B$. A simple volume argument shows that this would imply $|B\setminus C| \ge \frac{1}{2}|H'(B)|$, and as $B$ contains a cube of side length $2/\sqrt{n}$ we would have $$r^{-n}\epsilon_n(\tau) \ge |B\setminus C| \ge \frac{1}{2}|H'(B)|=\frac{1}{2}\left(\frac{t}{1-t}\right)^n|B|\ge \frac{1}{2} \left(\frac{\tau}{1-\tau}\right)^n\left(\frac{2}{\sqrt{n}}\right)^n.$$ Therefore as $b_n(\tau)^{-n} \ge r^{-n}$, taking $$\epsilon_n(\tau)<b_n(\tau)^n\frac{1}{2}\left(\frac{\tau}{1-\tau}\right)^n\left(\frac{2}{\sqrt{n}}\right)^n,$$
we deduce that $o \in D(C \cap B;t)$ and therefore in particular $x \in D(A';t)$.

\end{proof}

\begin{obs}
If $P$ is a (regular) simplex $T$, we can take $o$ to be the barycenter of $T$.
\end{obs}

\begin{proof}[Proof that \Cref{MAIN} implies \Cref{mainthm}]
We may assume that $t \le \frac{1}{2}$ since \Cref{mainthm} is invariant under replacing $t$ with $1-t$. By approximation, we can assume that $A$ has polyhedral convex hull $\co(A)$ with the vertices of $\co(A)$ lying in $A$ (see e.g. \cite[p.3 footnote 2]{Semisum}). 

Take $b_n(\tau)$ to be the minimum of $\tau$ and the constant such that $$\delta_n(\tau)^{-1}(1-(1-b_n(\tau))^n) = 1-C_n^{-1}\tau,$$ and take $\epsilon_n(\tau)$ as in \Cref{JohnClaim}.

From \Cref{BrunnMinkCor}, we see that we can choose $\Delta_n(\tau)$ sufficiently small so that $$|\co(A)\setminus A| \le \epsilon_n(\tau)|A|\le \epsilon_n(\tau)|\co(A)|,$$ and therefore by \Cref{JohnClaim} there is a translate of $(1-b_n(\tau))\co(A)\subset D(A;t)$. Let $o$ be the center of homothety relating this translate of $(1-b_n(\tau))\co(A)$ and $\co(A)$. Because $b_n(\tau)\le \tau$, the region $to+(1-t)\co(A)$ is contained in $D(A;t)$, so from this we deduce that $D(A \cup \{o\};t)=D(A;t)$. Therefore we may assume without loss of generality that $o \in A$.

Note that the inequality in \Cref{mainthm} that we want to deduce is equivalent to
\begin{align*}|\co(A)\setminus D(A;t)| \le& (1-C_n^{-1}\tau)|\co(A)\setminus A|.\end{align*} 

Triangulate $\co(A)$ into simplices $T_i$ by triangulating $\partial \co(A)$ and coning off each facet at $o$. Then in each simplex $T_i$, we claim that $$|T_i\setminus D(A;t)|\le(1-C_n^{-1}\tau)|T_i\setminus A|.$$

Provided $|T_i\setminus A| \le \delta_n(\tau)|T_i|$, applying \Cref{MAIN} to $T_i,A\cap T_i$ yields the stronger inequality
$$|T_i\setminus D(A\cap T_i;t)|\le (1-C_n^{-1}\tau)|T_i\setminus A|.$$
On the other hand, if $|T_i\setminus A| \ge \delta_n(\tau)|T_i|$, then as $b_n(\tau)o+(1-b_n(\tau))T_i \subset D(A;t)\cap T_i$, we have
$$|T_i \setminus D(A;t)|\le |T_i|(1-(1-b_n(\tau))^n) \le \delta_n(\tau)^{-1}(1-(1-b_n(\tau))^n)|T_i\setminus A| \le (1-C_n^{-1}\tau)|T_i\setminus A|.$$
We conclude by noting
$$|\co(A)\setminus D(A;t)| = \sum |T_i \setminus D(A;t)|\le \sum (1-C_n\tau^{-1})|T_i\setminus A|=(1-C_n^{-1}\tau)|\co(A)\setminus A|.$$
\end{proof}

\section{Setup and technical lemmas}

We take $A$ to satisfy the hypotheses of \Cref{MAIN}. We may assume that $t \le \frac{1}{2}$ since \Cref{MAIN} is invariant under replacing $t$ with $1-t$. It suffices to prove the statement for a particular choice of $T$ since all simplices of volume $1$ in $\mathbb{R}^n$ are equivalent under volume-preserving affine transformations. Hence we work in a fixed regular simplex $T\subset \mathbb{R}^n$ from now on.
Let $x_0,\ldots,x_n$ denote the vertices of $T$, and define the \emph{corner $\lambda^{i}$-scaled simplices} to be
$$S_i^j(\lambda)=(1-\lambda^i)x_j+ \lambda^iT \text{ for } 0 \le j \le n$$ and set $$\mathcal{S}_i(\lambda):=\{S_i^0(\lambda),\ldots,S_i^n(\lambda)\}.$$
In the picture below, we've shaded one of the $S^j_2(\frac{1}{2})$'s inside $T$ when $n=2$.
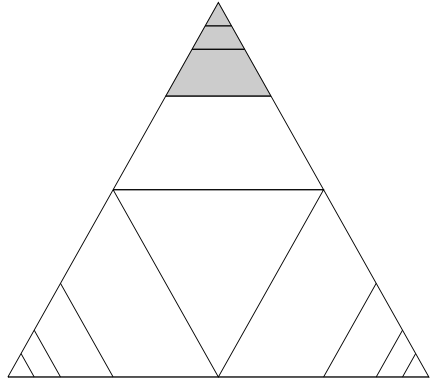
\begin{center}
\resizebox{5.77cm}{5.0cm}{
\begin{tikzpicture}
\draw (0,0) -- (4,4) -- (8,0) -- (0,0);
\draw (2,2)--(6,2)--(4,0)--(2,2);

\draw (3,3) -- (5,3);
\draw (3.5,3.5) -- (4.5,3.5);
\draw(3.75,3.75)--(4.25,3.75);

\draw (1,1)--(2,0);
\draw (0.5,0.5)--(1,0);
\draw (0.25,0.25)--(0.5,0);

\draw (7,1)--(6,0);
\draw (7.5,0.5)--(7,0);
\draw (7.75,0.25)--(7.5,0);

\draw[fill=black, opacity=0.2] (4,4)--(3,3)--(5,3)--cycle;
\end{tikzpicture}
}
\end{center}

Define the \emph{$\lambda^i$-scaled $k$-averaged simplices} $\mathcal{T}_{i,k}(\lambda)$ iteratively by
\begin{align*}
\mathcal{T}_{i,0}(\lambda)&=\mathcal{S}_i(\lambda)\\
\mathcal{T}_{i,k+1}(\lambda)&=\{\lambda B_1+(1-\lambda)B_2\mid B_1,B_2 \in \mathcal{T}_{i,k}(\lambda)\}.\end{align*}
Note that all simplices in $\mathcal{T}_{i,k}(\lambda)$ are translates of $\lambda^iT$, and we have the inclusions
$$\mathcal{T}_{i,0}(\lambda)\subset \mathcal{T}_{i,1}(\lambda)\subset  \mathcal{T}_{i,2}(\lambda) \subset \ldots$$
For fixed $i,\lambda$, the simplices in the family $\mathcal{T}_{i,k}(\lambda)$ eventually cover all of $T$ and heavily overlap each other as $k \to \infty$ (in fact the translates become dense among all possible translates of $\lambda^iT$ which lie inside $T$). Shaded below are the simplices in $\mathcal{T}_{2,1}(\frac{1}{2})$ when $n=2$.
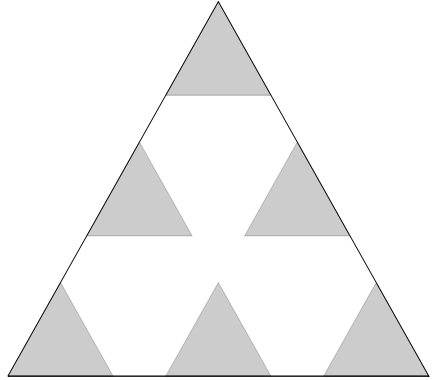
\begin{center}
\resizebox{5.77cm}{5.0cm}{
\begin{tikzpicture}
\draw (0,0) -- (4,4) -- (8,0) -- (0,0);
\draw[fill=black, opacity=0.2] (4,4)--(3,3)--(5,3)--cycle;
\draw[fill=black, opacity=0.2] (0,0)--(1,1)--(2,0)--cycle;
\draw[fill=black, opacity=0.2] (8,0)--(7,1)--(6,0)--cycle;

\draw[fill=black, opacity=0.2] (3,0)--(4,1)--(5,0)--cycle;
\draw[fill=black, opacity=0.2] (1.5,1.5)--(2.5,2.5)--(3.5,1.5)--cycle;
\draw[fill=black, opacity=0.2] (4.5,1.5)--(5.5,2.5)--(6.5,1.5)--cycle;
\end{tikzpicture}
}
\end{center}

\Cref{reflem} is the crux of our argument. The proof of \Cref{reflem} shows that  for all $T' \in \mathcal{T}_{i,k}(1-t)$, the set $|T'\cap A|$ contains a translated copy of $(1-t)^iA$ (up to a bounded error). This fractal structure allows us to conclude that $|T'\cap A|$ is bounded below by $|T'|(1-\delta)$ (up to a bounded error).

\begin{lem}
\label{reflem}The constants $c_{i,k,n}=i+2k$ are such that for every $T' \in \mathcal{T}_{i,k}(1-t)$ we have
$$|T'\cap A| \ge |T'|(1-\delta)-c_{i,k,n}\delta(A;t).$$
\end{lem}
\begin{proof}
For the remainder of this proof, we will denote
$$\lambda=1-t,$$ and write for notational convenience $S_i^j$ instead of $S_i^j(\lambda)$. The following notation will be useful for us: consider the translation that brings $\lambda^iT$ to $T'$ and denote by $(\lambda^iA)_{T'}$ the shift of the set $\lambda^iA$ under this translation.



We shall actually show the stronger inequalities
$$|(\lambda^iA)_{T'} \setminus A |\le c_{i,k,n}\delta(A;t)$$
(which are stronger as $|(\lambda^i A)_{T'}|=|T'|(1-\delta)$).


First, we show the inequality when $k=0$.  Recall that if $T' \in \mathcal{T}_{i,0}(\lambda)$ then $T'=S^j_i$ for some $j$. The inequality is trivial for $(i,k)=(0,0)$ by definition of $\delta$.

We now show the inequality for $(i,k)=(1,0)$. Note $(\lambda A)_{S^j_1}=(1-\lambda) x_j+\lambda A\subset D(A;t)$, so
$$|(\lambda A)_{S^j_1}\setminus A| \le |D(A;t)\setminus A|=\delta(A;t).$$

Suppose we know the result for $(i,0)$, we now prove the result for $(i+1,0)$. Then $(\lambda^{i+1}A)_{S^j_{i+1}}=(1-\lambda^{i+1})x_j+\lambda^{i+1}A$, and we have 
\begin{align*}|(\lambda^{i+1}A)_{S^j_{i+1}}\setminus A|&\le |(\lambda^{i+1}A)_{S^j_{i+1}}\setminus (\lambda A)_{S^j_1}|+|(\lambda A)_{S^j_1}\setminus A|\\
&=\lambda^n|(\lambda^i A)_{S^j_i}\setminus A|+|(\lambda A)_{S_1^j}\setminus A|\\
&\le (\lambda^n c_{i,0,n}+c_{1,0,n})\delta(A;t)
\\
&\le c_{i+1,0,n}\delta(A;t).\end{align*}

Finally, we induct on $k$. We have proved the base case $k=0$, so  assume the inequality for $(i,k)$. We will now prove the inequality for $(i,k+1)$.

Thus we suppose that $T' \in \mathcal{T}_{i,k+1}$, which by definition means that there exists $T'_1,T'_2 \in \mathcal{T}_{i,k}$ such that
$$T'=\lambda T'_1+(1-\lambda) T'_2.$$
We now prove an easy claim before returning to the proof of the lemma.
\begin{claim}
Let $X,X'$ be translates of each other in $\mathbb{R}^n$ with common volume $V=|X|=|X'|$, and let $Y \subset X$, $Y' \subset X'$. Then if $V'$ is a constant such that $|X \setminus Y|,|X'\setminus Y'|\le V'$, we have
$$|\lambda Y+(1-\lambda)Y'| \ge V-V'.$$
\end{claim}
\begin{proof}
We have $|Y|,|Y'| \ge V-V'$, so the result follows from the Brunn-Minkowski inequality.
\end{proof}
Returning to the proof of the lemma, we have by the induction hypothesis that both
\begin{align*}
&| (\lambda^i A)_{T_1'}\setminus A| \le c_{i,k,n}\delta(A;t),\text{ and}\\
&|(\lambda^i A)_{T_2'}\setminus A| \le c_{i,k,n}\delta(A;t).
\end{align*}
Because $(\lambda^i A)_{T_1'}$ and  $(\lambda^i A)_{T_2'}$ are translates of each other with common volume $(1-\delta)|T'|$, setting $X=(\lambda^i A)_{T_1'}$, $X'=(\lambda^i A)_{T_2'}$, $Y=A\cap (\lambda^i A)_{T_1'}$, $Y'=A\cap (\lambda^i A)_{T_2'}$ we deduce from the claim that
$$|\lambda(A \cap (\lambda^i A)_{T_1'})+(1-\lambda)(A \cap (\lambda^i A)_{T_2'})| \ge |T'|(1-\delta) - c_{i,k,n}\delta(A;t).$$
Because $D(A;t)=\lambda A+(1-\lambda)A$ and $(\lambda^i D(A;t))_{T'}=\lambda(\lambda^i A)_{T_1'}+(1-\lambda)(\lambda^i A)_{T_2'}$, we have
\begin{align*}
|D(A;t)\cap (\lambda^i A)_{T'}|&\ge|D(A;t) \cap  (\lambda^i D(A;t))_{T'}|-|\lambda^iD(A;t)\setminus \lambda^iA|\\
&\ge|D(A;t) \cap  (\lambda^i D(A;t))_{T'}|-\delta(A;t)\\
 &\ge |\lambda(A \cap (\lambda^i A)_{T_1'})+(1-\lambda)(A \cap (\lambda^i A)_{T_2'})|-\delta(A;t)
\\&\ge  |T'|(1-\delta) - (c_{i,k,n}+1)\delta(A;t),
\end{align*}
which as $|(\lambda^i A)_{T'}|=(1-\delta)|T'|$ is equivalent to
$$|(\lambda^i A)_{T'}\setminus D(A;t)| \le (c_{i,k,n}+1)\delta(A;t).$$
We conclude that 
\begin{align*}|(\lambda^i A)_{T'}\setminus A| &\le |(\lambda^i A)_{T'}\setminus D(A;t)| + \delta(A;t) \\&\le (c_{i,k,n}+2)\delta(A;t)
\\&= c_{i,k+1,n}\delta(A;t).\end{align*}
\end{proof}


The following lemma shows that given $\alpha<1$ and $\frac{1}{2}\le \lambda<1$, any arbitrary covering of $T$ by translates of $\alpha^n\lambda T$ contained inside $T$ can be approximated by a covering consisting of elements of $\mathcal{T}_{i,k}(\lambda)$ for fixed small values $i,k$. The parameters $i,k$ are positively correlated with $\lambda, \alpha$.

Before we proceed, we need the following notation. Let
$\mathcal{T}_k(\lambda;\lambda';T)$ be recursively defined by setting
\begin{align*}
    \mathcal{T}_0(\lambda;\lambda';T)&=\{\lambda' T+(1-\lambda')x_j \mid j \in \{0,\ldots,n\}\}\\
    \mathcal{T}_k(\lambda;\lambda';T)&=\{\lambda B_1 +(1-\lambda)B_2 \mid B_1,B_2 \in \mathcal{T}_{k-1}(\lambda;\lambda';T)\}.
\end{align*}
Note that by definition, $\mathcal{T}_{i,k}(\lambda)=\mathcal{T}_k(\lambda;\lambda^i;T)$. 


\begin{lem}\label{usefullem} 
For $\alpha,\mu\in (0,1), \lambda\in[\frac12,1)$, every translate $T'\subset T$ of $\alpha^n \mu T$ is completely contained in some element of $\mathcal{T}_{k'}(\lambda;\mu;T)$ with 
$$k'=\sum_{j=1}^n \lceil \log(\alpha^{j-1}(1-\alpha)\mu)/\log(\lambda)\rceil.$$
\end{lem}

\begin{proof}
To prove this we need the following claim, which is essentially the result for $n=1$.
\begin{claim}
Every weighted average of two (corner) simplices in $\mathcal{T}_0(\lambda;\alpha\mu;T)$ lies in some simplex of $\mathcal{T}_\ell(\lambda;\mu;T)$ with $\ell=\lceil\log((1-\alpha)\mu )/\log(\lambda)\rceil$
\end{claim}

\begin{proof}
Suppose the two corner simplices are at the corners $x_{a}$ and $x_{b}$. Then every homothetic copy $T'\subset T$ of $T$ is determined by the corresponding edge $x_a'x_b'$. Thus the claim is implied by the one-dimensional version of the claim by intersecting all simplices with $x_ax_b$. Hence we may assume that $T=[0,1]$, so that $\mathcal{T}_{0}(\lambda;\mu;T)=\{[0,\mu],[1-\mu,1]\}$, and we want to show that every sub-interval of $[0,1]$ of length $\alpha\mu$ is contained in an element of $\mathcal{T}_{\ell}(\lambda;\mu;T)$.

We will now proceed by showing that the largest distance between consecutive midpoints of intervals in $\mathcal{T}_{j+1}(\lambda;\mu;T)$ is at most $\lambda$ times the largest such distance in $\mathcal{T}_j(\lambda;\mu;T)$. Let $I_1,I_2$ be two consecutive intervals in  $\mathcal{T}_{j}(\lambda;\mu;T)$ for some $j$. Then in $\mathcal{T}_{j+1}(\lambda;\mu;T)$ we also have the intervals $J=\lambda I_1+(1-\lambda) I_2$ and $K=(1-\lambda) I_1+\lambda I_2$, and the intervals $I_1,J,K,I_2$ appear in this order from left to right as $\lambda\ge \frac{1}{2}$. If $d$ is the distance between the midpoints of $I_1,I_2$, then the distances between the consecutive midpoints of $I_1,J,K,I_2$ are $(1-\lambda)d, (2\lambda-1)d, (1-\lambda)d$ respectively. Therefore, the largest distance between two midpoints $d_{j+1}$ in $\mathcal{T}_{j+1}(\lambda;\mu;T)$ is at most $\max(1-\lambda,2\lambda-1,1-\lambda)d_j\le \lambda d_j$ where $d_j$ is the largest distance between two consecutive midpoints in $\mathcal{T}_{j}(\lambda;\mu;T)$. Therefore, the distance between two consecutive midpoints in $\mathcal{T}_{\ell}(\lambda;\mu;T)$ is at most $\lambda^{\ell}\leq (1-\alpha)\mu$.

Given an interval $I$ of length $\alpha \mu$, then either the midpoint lies in $[0,\mu/2]\cup [1-\mu/2,1]$, in which case $I$ is already contained in one of $[0,\mu]$ or $[1-\mu,1]$ belonging to $\mathcal{T}_0(\lambda;\mu;T)$, or else we can find an interval $I'\in\mathcal{T}_{\ell}(\lambda;\mu;T)$ of length $\mu$ such that the distance between the midpoints of $I$ and $I'$ is at most $\frac{1}{2}(1-\alpha)\mu$, which implies $I\subset I'$.
\end{proof}

We prove our desired statement by induction on the dimension $n$. The claim above proves the base case $n=1$, so now assuming the statement is true for dimensions up to $n-1$, we will show it to be true for $n$.

Let $T'\subset T$ be a fixed translate of $\alpha^n\mu T$, with corresponding vertices $x_0',\ldots,x_n'$. Denote by $F$ the facet of $T$ opposite $x_{n}$, and denote by $F'$ the facet of $T'$ opposite the corresponding vertex $x_n'$. Denote by $H$ the hyperplane spanned by $F'$. Then $S=H\cap T$ is an $n-1$-simplex, with vertices $y_0,\ldots,y_{n-1}$ such that $y_i$ is on the edge of $T$ connecting $x_i$ to $x_n$.

If the common ratio $r:=|y_jx_n|/|x_jx_n| \le \alpha\mu$, then $T'$ is already contained in an element of $\mathcal{T}_0(\lambda;\mu;T)$ and we are done. Otherwise, denote by $T_0,\ldots,T_{n-1}\subset T$ the translates of $\alpha\mu T$ that sit on $H$ and have corners at $y_0,\ldots,y_{n-1}$ respectively. Denote the facet $T_i\cap H$ of $T_i$ by $F_i$. We remark that each $F_i$ is a translate of $\mu' S$ for some fixed $\mu' \ge \alpha\mu$.

By the claim, the simplices $T_0,\ldots,T_{n-1}$ are completely contained in elements of $\mathcal{T}_{\ell}(\lambda;\mu;T)$ with
$$\ell=\lceil\log((1-\alpha)\mu)/\log(\lambda)\rceil.$$

By the induction hypothesis applied to the $n-1$-simplex $S$, $F'$ is completely contained in a simplex from the family $\mathcal{T}_{\ell'}(\lambda;\mu';S)$ for $$\ell':= \sum_{j=1}^{n-1} \lceil\log((1-\alpha)\alpha^{j-1}\mu')/\log(\lambda)\rceil\leq \sum_{j=1}^{n-1} \lceil\log((1-\alpha)\alpha^{j}\mu)/\log(\lambda)\rceil ,$$
as $\mu' \ge \alpha\mu$.
Note that $F'$ is contained in a certain iterated weighted average of the facets $F_0,\ldots,F_{n-1}$ if and only if $T'$ is contained in the analogously defined iterated weighted average of $T_0,\ldots,T_{n-1}$. Therefore
$T' \in \mathcal{T}_{\ell+\ell'}(\lambda;\mu;T)$.

Finally, we have that $\ell+\ell' \le k'$, so $T' \in \mathcal{T}_{k'}(\lambda;\mu;T)$ as desired.
\end{proof}






The following lemma helps to show that arbitrary coverings of $T$ can be modified at no extra cost to coverings of $T$ contained inside $T$. 

\begin{lem}\label{homosimplex}
Let $r\in (0,1)$ and $rT+x$ a translate of $rT$. Then there exists a $y$ such that $(rT+x)\cap T \subset rT+y \subset T$.
\end{lem}
\begin{proof}
The intersection of any two copies of the simplex $T$ is itself homothetic to $T$. Therefore $(rT+x)\cap T$ is homothetic to $T$, and so must be a translate of $r'T$ for some $r' \le r$. Because $T$ is convex and $(rT+x) \cap T$ is a homothetic copy of $T$ lying inside $T$, the center of homothety between $(rT+x)\cap T$ and $T$ lies inside $(rT+x)\cap T$, and all intermediate homotheties lie inside $T$. In particular there is a homothety which produces a translate of $rT$ which lies inside $T$, and this translate by construction contains $(rT+x)\cap T$.
\end{proof}

\section{Proof of \Cref{MAIN}}
Recall that we may assume that $t \le \frac{1}{2}$.

\begin{proof}[Proof of \Cref{MAIN}]
Let $i=\left\lceil\log\left(\frac{n^{\frac{1}{n}}}{(2n)^5}\right)/\log(1-t)\right\rceil$, so that $(1-t)^i \in \left[\frac{n^{\frac{1}{n}}}{2(2n)^5},\frac{n^{\frac{1}{n}}}{(2n)^5}\right]$ (as $t \le \frac{1}{2}$). Note that
\begin{align*}
    i \le 1+\log((2n)^5)/t\le 6\log(2n)/t.
\end{align*}
Let $\eta=n^{-\frac{1}{n}}(1-t)^i\in [\frac{1}{2(2n)^5},\frac{1}{ (2n)^5}]$, and let $\zeta=(n+1)\eta$.

Recall $T$ is a regular simplex of volume $1$, denote by $o$ the barycenter. By \Cref{JohnClaim}, setting $o$ to be the origin, if we choose $\delta_n(\tau)$ sufficiently small, then $R:=(1-\zeta)T$ is contained in $D(A;t)$.
Let $L=(1+\eta(n+1))T\setminus (1-\zeta-\eta(n+1))T$. Note that $T\setminus R\subset L$ and for any $T'$ a translate of $\eta T$ intersecting $T\setminus R$, we have $T'\subset L$.

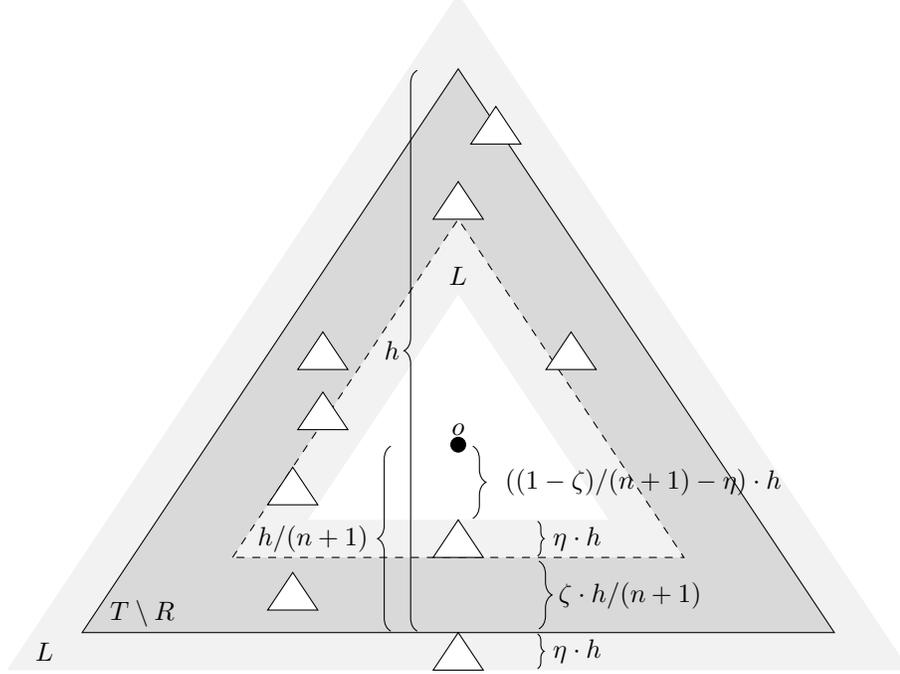
\begin{center}\begin{tikzpicture}[x=1cm,y=1cm]
\fill[fill=gray!10] (-2,-1)--(4,8)--(10,-1)--cycle;
\draw[fill=gray!30] (-1,-0.5)--(4,7)--(9,-0.5)--cycle;
\fill[fill=gray!10] (1,0.5)--(4,5)--(7,0.5)--cycle;
\draw[dashed] (1,0.5)--(4,5)--(7,0.5)--cycle;
\fill[fill=white] (2,1)--(4,4)--(6,1)--cycle;
\draw[fill=white] ({4-(0.5)*(2/3)},-1)--(4,-0.5)--({4+(0.5)*(2/3)},-1)--cycle;
\draw[fill=white] ({4-(0.5)*(2/3)},0.5)--(4,1)--({4+(0.5)*(2/3)},0.5)--cycle;
\draw[decorate,decoration={brace,amplitude=5pt,raise=0.5pt},xshift=30pt] (4,0.5-0.05) -- (4,-0.5+0.05) node [midway,yshift=0pt, xshift =35pt]{$\zeta \cdot h/(n+1)$};
\draw[decorate,decoration={brace},xshift=30pt] (4,-0.5-0.02) -- (4,-1+0.02) node [midway,yshift=0pt, xshift =15pt]{$\eta \cdot h$};
\draw[decorate,decoration={brace},xshift=30pt] (4,1.0-0.02) -- (4,0.5+0.02) node [midway,yshift=0pt, xshift =15pt]{$\eta \cdot h$};
\draw[decorate,decoration={brace,amplitude=5pt,raise=0.5pt},xshift=5pt] (4,2.0-0.02) -- (4,1.0+0.02) node [midway,yshift=0pt, xshift =65pt]{$((1-\zeta)/(n+1)-\eta)\cdot h$};
\draw[decorate,decoration={brace,amplitude=5pt,raise=0.5pt,mirror},xshift=-25pt] (4,2.0-0.02) -- (4,-0.5+0.02) node [midway,yshift=0pt, xshift =-30pt]{$h/(n+1)$};
\draw[decorate,decoration={brace,amplitude=5pt,raise=0.5pt,mirror},xshift=-15pt] (4,7.0-0.02) -- (4,-0.5+0.02) node [midway,yshift=0pt, xshift =-10pt]{$h$};

\draw (-0.20,-0.25) node {$T\setminus R$};
\draw (-1.5,-0.75) node {$L$};
\draw (4,4.25) node {$L$};

\draw[fill=white] ({4-(0.5)*(2/3)-2.2},{0.5-0.7})--(4-2.2,{1-0.7})--({4+(0.5)*(2/3)-2.2},{0.5-0.7})--cycle;
\draw[fill=white] ({4-(0.5)*(2/3)-2.2},{0.5+0.7})--(4-2.2,{1+0.7})--({4+(0.5)*(2/3)-2.2},{0.5+0.7})--cycle;
\draw[fill=white] ({4-(0.5)*(2/3)-1.8},{0.5+1.7})--(4-1.8,{1+1.7})--({4+(0.5)*(2/3)-1.8},{0.5+1.7})--cycle;
\draw[fill=white] ({4-(0.5)*(2/3)-1.8},{0.5+2.5})--(4-1.8,{1+2.5})--({4+(0.5)*(2/3)-1.8},{0.5+2.5})--cycle;
\draw[fill=white] ({4-(0.5)*(2/3)},{0.5+4.5})--(4,{1+4.5})--({4+(0.5)*(2/3)},{0.5+4.5})--cycle;
\draw[fill=white] ({4-(0.5)*(2/3)+0.5},{0.5+5.5})--(4+0.5,{1+5.5})--({4+(0.5)*(2/3)+0.5},{0.5+5.5})--cycle;
\draw[fill=white] ({4-(0.5)*(2/3)+1.5},{0.5+2.5})--(4+1.5,{1+2.5})--({4+(0.5)*(2/3)+1.5},{0.5+2.5})--cycle;

\fill (4,2) circle (3pt);
\draw (4,2) node[anchor=south]{$o$};

\end{tikzpicture}
\end{center}

\begin{claim}
There exists a covering $\mathcal{B}$ of $T\setminus R$ by translates of $\eta T$ contained in $T$, such that $\sum_{T'\in\mathcal{B}} |T'|\le \frac{1}{2n}$ and $|\mathcal{B}|\leq (2n)^{5n}$.
\end{claim}

\begin{proof}[Proof of claim]
It follows from \cite{Rogers}  that\footnote{We note that $n=2$ is not mentioned explicitly in \cite{Rogers} but follows easily.} for all $n\geq2$, there exists $r\in\mathbb{R}$ and there exists a covering $\mathcal{F}$ of $\mathbb{R}^n / (r\mathbb{Z})^n$ by translates of $\eta T$ with average density at most $7n\log(n)$, i.e.
$$\frac{|\mathcal{F}|\eta^n}{r^n} \le 7n \log n.$$
Passing to a multiple of $r$, we may assume that $T,L\subset[-r/2,r/2]^n$.  Consider a uniformly random translate $\mathcal{F}+\bf{x}$. For any $T' \in \mathcal{F}$ and any point $t' \in T'$, we have $$\mathbb{P}(T'+{\bf{x}} \subset L) \le \mathbb{P}(t'+{\bf{x}} \in L) =  \frac{|L|}{r^n}.$$
 Therefore, $$\mathbb{E}(|\{T'+{\bf{x}}\in \mathcal{F}+{\bf{x}} \mid  T'+{\bf{x}} \subset L\}|)\le \frac{|L|}{r^n}|\mathcal{F}|\le |L|\eta^{-n}7n\log n,$$ so there exists an $\bf{x_0}$ such that
$$|\{T'+{\bf{x_0}} \in \mathcal{F}+{\bf{x_0}} \mid  T'+{\bf{x_0}} \subset L\}| \le |L|\eta^{-n}7n\log n.$$
Define $$\mathcal{B}'=\{T'+{\bf{x_0}}\in \mathcal{F}+{\bf{x_0}} \mid (T'+{\bf{x_0}}) \cap (T\setminus R) \ne \emptyset\},$$ then by the above discussion we have $\mathcal{B}'$ is a covering of $T\setminus R$, and $$|\mathcal{B}'|\le |L|\eta^{-n}7n\log n.$$
By \Cref{homosimplex}, for each element $T' + {\bf{x_0}}\in \mathcal{B}'$ we can find a translate $T'+{\bf{y_{T'}}} $ such that $(T'+{\bf{x_0}})\cap T\subset T'+{\bf{y_{T'}}}\subset T$. Define $$\mathcal{B}=\{T'+{\bf{y_{T'}}}\mid T'+{\bf{x_0}}\in \mathcal{B}'\},$$
then $\mathcal{B}$ is a cover of $T\setminus R$ by translate of $\eta T$ contained in $T$ with $|\mathcal{B}| \le |L| \eta^{-n}7n\log n$.

We can calculate the upper bound
\begin{align*}
    |L|&=(1+\eta(n+1))^n-(1-\zeta - \eta(n+1))^n\\
    &= (1+\eta(n+1))^n-(1- 2\eta(n+1))^n\\
    &\le 1+2\eta n(n+1)-(1-2\eta n(n+1))\\
    &=4\eta n(n+1).
\end{align*}
The inequality follows from the fact that $\eta^k(n+1)^k \binom{n}{k}\le (1/2)^k 2\eta(n+1)$ and the convexity of $(1-x)^n$ for $x\in (0,1)$.

Therefore,
$$|\mathcal{B}|\le 4\eta n(n+1)\eta^{-n} (7n \log n) \le \frac{1}{2n}\eta^{-n} \le (2n)^{5n}$$
and
$$\sum_{T' \in \mathcal{B}} |T'| = \eta^n |\mathcal{B}|\le \eta^n\frac{1}{2n}\eta^{-n}\le \frac{1}{2n}.$$
\end{proof}

\begin{claim}
There is a cover $\mathcal{A}\subset \mathcal{T}_{i,k}(1-t)$ of $T\setminus R$ with $k \le 8n\log(2n)/t$ such that $|\mathcal{A}|\le (2n)^{5n}$ and $\sum_{T''\in \mathcal{A}}|T''|\le \frac{1}{2}$.
\end{claim}
\begin{proof}
We apply \Cref{usefullem} with $\alpha=n^{-\frac{1}{n}}$, $\lambda=1-t$, $\mu=(1-t)^i$.  \begin{align*}k=\sum_{j=1}^n \lceil\log(\alpha^{j-1}(1-\alpha)\mu)/\log(\lambda)\rceil&\le n \lceil \log (\alpha^n(1-\alpha) \mu)/\log(\lambda)\rceil\\
&\le n\left(1+\log\left(\frac{\log  n}{2n^2}\mu\right)/\log(\lambda)\right)\\
&\le n\left(1+\log\left(\frac{\log n}{(2n)^7}\right)/(-t)\right)\\
&\le 8n\log(2n)/t.\end{align*} 
This shows that every translate of $\eta T=n^{-\frac{1}{n}}(1-t)^iT$ inside $T$ is  contained in some element of $\mathcal{T}_k((1-t);(1-t)^i;T)=\mathcal{T}_{i,k}(1-t)$. For each simplex $T' \in \mathcal{B}$, we can therefore choose a simplex $f(T') \in \mathcal{T}_{i,k}(1-t)$ such that $T'\subset f(T')$. Let
$$\mathcal{A}=\{f(T') \mid T' \in \mathcal{B}\}.$$
Note that $\mathcal{A}$ is a cover of $T\setminus R$,
$$|\mathcal{A}|=|\mathcal{B}| \le (2n)^{5n},$$
and
$$\sum_{T'' \in \mathcal{A}}|T''|=(n^{\frac{1}{n}})^n\sum_{T' \in \mathcal{B}}|T'| \le \frac{1}{2}.$$
\end{proof}

Returning to the proof of \Cref{MAIN}, note that since $\mathcal{A}\subset \mathcal{T}_{i,k}(1-t)$,  \Cref{reflem} implies that for every $T'' \in \mathcal{A}$ we have $$|T''\setminus A| \le \frac{|T\setminus A|}{|T|}|T''|+c_{i,k,n}\delta(A;t).$$ Since $R\subset D(A;t)$, we have
\begin{align*}
    |T\setminus D(A;t)| &= |(T\setminus R)\setminus D(A;t)|
    \le \sum_{T'' \in \mathcal{A}}|T'' \setminus D(A;t)|\le \sum_{T'' \in \mathcal{A}}|T''\setminus A|\\
    &\le \frac{|T\setminus A|}{|T|}\sum_{T''\in \mathcal{A}}|T''|+ |\mathcal{A}|\cdot c_{i,k,n}\delta(A;t)
    \le  \frac{1}{2}|T\setminus A|+|\mathcal{A}|\cdot c_{i,k,n}\delta(A;t),
\end{align*}
which after replacing $|T\setminus D(A;t)|=|T\setminus A|-\delta(A;t)$ yields
$$|T\setminus A| \le 2(1+|\mathcal{A}|\cdot c_{i,k,n})\delta(A;t).$$ 
We estimate
\begin{align*}
c_{i,k,n}=i+2k\le 6\log(2n)/t+16n\log(2n)/t\le 19 n \log(2n)/t
\end{align*}
Therefore 
$$2(1+|\mathcal{A}|c_{i,k,n})\le 2(1+(2n)^{5n}(19 n \log(2n)/t)) \le (4n)^{5n}/\tau.$$
In conclusion, with $C_n=(4n)^{5n}$ we obtain
$$|T\setminus A| \le C_n \tau^{-1} \delta(A;t)$$
as desired.
\end{proof}

\section{Sharpness of $C_n$}
\label{CnSharp}
In studying the asymptotic behaviour of the optimal value of $C_n$ in \Cref{mainthm}, we note that there is still a gap of order $\log(n)$ in the exponent between the upper and lower bounds. Our proof shows the upper bound $C_n \le (4n)^{5n}=e^{5n\log(4n)}$ and, the example mentioned in the introduction shows the lower bound $C_n \ge \frac{2^{n-1}}{n}$.

In our method the complexity of $C_n$ is limited by the fact that  $|\mathcal{A}|\leq C_n$, where $\mathcal{A}$ is a set of translates of $\eta T$ contained inside $T$ with $\eta \le \frac{1}{2}$ covering $\partial T$ and satisfying $\sum_{T' \in \mathcal{A}}|T'| <|T|$. In fact, by a slight restructuring of our proof it is equivalent to covering just a single facet $F$ of $T$. Taking $\mathcal{A}'$ to be the family of intersections of elements of $\mathcal{A}$ with the hyperplane containing $F$, we see that $|\mathcal{A}'|\leq C_n$ with $\mathcal{A}'$ a set of translates of $\eta F$ covering $F$ and $\sum_{F' \in \mathcal{A}'} |F'|< \eta^{-1}|F|$.

\begin{quest}
Is it true that for every $0<\eta_0\leq \frac12$, then for all sufficiently large $n$ if $F\subset\mathbb{R}^n$ is a simplex and $\mathcal{A}'$ is a family of translates of $\eta_0 F$ covering $F$ we have
$$\sum_{F' \in \mathcal{A}'} |F'|> \eta_0^{-1}|F|?$$
\end{quest}
Resolving this question would shed light on the correct growth rate of $C_n$. In particular, if the question has a negative answer with $\eta_0^{-1}$ replaced with $\eta_0^{-1}(1-\epsilon)$ for some fixed $\epsilon$, then our methods would show that $C_n$ has exponential growth.

\bibliographystyle{abbrv}
\bibliography{references}

\begin{thebibliography}{10}

\bibitem{Barchiesi}
M.~Barchiesi and V.~Julin.
\newblock Robustness of the gaussian concentration inequality and the
  {B}runn-{M}inkowski inequality.
\newblock {\em Calc. Var. Partial Differential Equations}, 56, 05 2017.

\bibitem{Oneconvex}
E.~Carlen and F.~Maggi.
\newblock Stability for the {B}runn-{M}inkowski and {R}iesz rearrangement
  inequalities, with applications to gaussian concentration and finite range
  non-local isoperimetry.
\newblock {\em Canadian Journal of Mathematics}, 69, 07 2015.

\bibitem{Christ}
M.~Christ.
\newblock Near equality in the {B}runn-{M}inkowski inequality.
\newblock {\em {\tt arXiv:1207.5062}}, 2012.

\bibitem{Fig}
A.~Figalli and D.~Jerison.
\newblock Quantitative stability for sumsets in {$\Bbb{R}^n$}.
\newblock {\em J. Eur. Math. Soc. (JEMS)}, 17(5):1079--1106, 2015.

\bibitem{BM}
A.~Figalli and D.~Jerison.
\newblock Quantitative stability for the {B}runn-{M}inkowski inequality.
\newblock {\em Adv. Math.}, 314:1--47, 2017.

\bibitem{Semisum}
A.~Figalli and D.~Jerison.
\newblock A sharp {F}reiman type estimate for semisums in two and three
  dimensional euclidean spaces.
\newblock {\em Ann. Sci. Ec. Norm. Supr.}, 2019.

\bibitem{Euclidean}
A.~Figalli, F.~Maggi, and C.~Mooney.
\newblock The sharp quantitative {E}uclidean concentration inequality.
\newblock {\em Cambridge Journal of Mathemtics}, 6:59--87, 3 2018.

\bibitem{Figalli09}
A.~Figalli, F.~Maggi, and A.~Pratelli.
\newblock A refined {B}runn-{M}inkowski inequality for convex sets.
\newblock {\em Annales de l'Institut Henri Poincare (C) Non Linear Analysis},
  26:2511--2519, 11 2009.

\bibitem{Figalli10amass}
A.~Figalli, F.~Maggi, and A.~Pratelli.
\newblock A mass transportation approach to quantitative isoperimetric
  inequalities.
\newblock {\em Invent. Math}, pages 167--211, 2010.

\bibitem{Freiman}
G.~A. Fre\u{\i}man.
\newblock The addition of finite sets. {I}.
\newblock {\em Izv. Vys\v{s}. U\v{c}ebn. Zaved. Matematika}, 1959(6
  (13)):202--213, 1959.

\bibitem{John}
F.~John.
\newblock Extremum problems with inequalities as subsidiary conditions.
\newblock {\em Studies and Essays Presented to Courant on his 60th Birthday},
  pages 187--204, January 8, 1948.

\bibitem{Rogers}
C.~Rogers.
\newblock A note on coverings.
\newblock {\em Mathematika}, 4(1):1--6, 1957.

\end{thebibliography}

\end{document}